\documentclass{amsart}
\usepackage{amsfonts,amsmath, amsthm,amssymb}
\usepackage{latexsym}
\usepackage[all]{xy}
\usepackage{hyperref}
\usepackage{a4wide}

\theoremstyle{plain}
\newtheorem{theorem}{Theorem}
\newtheorem*{thm}{Theorem}
\newtheorem*{thma}{Theorem 10 bis}
\newtheorem{proposition}[theorem]{Proposition}
\newtheorem{corollary}[theorem]{Corollary}
\newtheorem{lemma}[theorem]{Lemma}
\newtheorem*{main question}{Question}
\newtheorem{question}{Question}
\newtheorem*{example}{Example}

\theoremstyle{definition}
\newtheorem{definition}[theorem]{Definition}
\newtheorem*{remark}{Remark}
\newtheorem*{remarks}{Remarks}

\DeclareMathOperator{\HH}{H}
\renewcommand{\H}{\HH}
\DeclareMathOperator{\Gal}{Gal}
\DeclareMathOperator{\Hom}{Hom}

\DeclareMathOperator{\im}{im}

\def\R{{\mathbb R}}
\def\C{{\mathbb C}}
\def\F{{\mathbb F}}
\def\G{{\mathbb G}}

\def\Q{{\mathbb Q}}
\def\Z{{\mathbb Z}}
\DeclareMathOperator{\princ}{princ}

\def\NDT{{Nguy\^e\~n Duy T\^an}}

\begin{document}
\title[Galois cohomology of unipotent groups]{On Galois cohomology of unipotent algebraic groups over local fields} 
 \author{ \NDT }
 \address{ 
Universit\"at Duisburg-Essen, FB6, Mathematik, 45117 Essen, Germany and Institute of Mathematics, 18 Hoang Quoc Viet, 10307, Hanoi.  Vietnam 
} 
\email{duy-tan.nguyen@uni-due.de}
\thanks{Partially supported by the NAFOSTED and the ERC/Advanced Grant 226257}

\date{}         

\begin{abstract} In this paper, we give a necessary and sufficient condition for the finiteness of Galois cohomology of  unipotent groups  over local fields of positive characteristic.

\end{abstract}
\maketitle 

\section*{Introduction}
Let $k$ be a field, $G$ a linear algebraic $k$-group. We denote by $k_s$ the separable closure of $k$ in an algebraic closure $\bar{k}$, by $\H^1(k,G):=\H^1(\Gal(k_s/k),G(k_s))$ the usual first Galois cohomology set. It is important to know the finiteness of the Galois cohomology set of algebraic groups over certain arithmetic fields such as local or global fields. Here, by a global field we mean either of the following: an algebraic number field, i.e., a finite extension of $\Q$; or a global function field, i.e., the function field of an algebraic curve over a finite field, equivalently, a finite extension of $\F_q(t)$, the field of rational functions in one variable over the finite field with $q$ elements. By a local field, we mean the completion of a global field at one of its place. 
So, a local field of characteristic $0$  is the real numbers $\R$, or the complex numbers $\C$, or a finite extension of $\Q_p$, the field of $p$-adic numbers, and a local field of characteristic $p>0$ is a finite extension of $\F_q((t))$, the field of formal Laurent series over $\F_q$. The following result is well known.

\begin{thm}
\mbox{ }
\begin{enumerate}
\item[(a)] (Borel-Serre) Let $k$ be a local field of characteristic 0, $G$ a linear algebraic group defined over $k$. Then $\H^1(k,G)$ is finite.
\item[(b)] (J.Tits) Let $k$ be a local field of positive characteristic, $G$ a connected reductive group defined over $k$. Then $\H^1(k,G)$ is finite.
\end{enumerate}
\end{thm}
For (a), we refer readers to \cite[Chapter III, Section 4, Theorem 4]{Se1}, and for (b), to \cite[Chapter III, Section 4, Remarks, p. 146]{Se1} or \cite[Proposition 7.2]{GiMB}. Oesterl\'e also gives examples of connected unipotent groups with infinite Galois cohomology over local fields of positive characteristic (see \cite[p. 45]{Oe}). 
\begin{example} [J. Oesterl\'e] Let $k=\F_q((t))$ of characteristic $p>2$, and  $G=\{(x,y)\in \G_a\times \G_a \mid y^p-tx^p-x=0\}$. Then $G$ is a  connected unipotent $k$-group and $\H^1(k,G)$ is infinite.
\end{example}

The following question is quite naturally arisen.  
\begin{main question} Let $k$ be a local field of positive characteristic, $G$ a unipotent group defined over $k$. When is $\H^1(k,G)$ finite?
\end{main question}

 In this article, we will give a complete answer for this question. 

In Section 1, first we present several technical lemmas concerning the images of additive polynomials in local fields which are needed in the sequel,  and  then we introduce the notion of unipotent groups of Rosenlicht type, see Definition~\ref{def:Rosenlicht type}. In Section 2, we apply the results of Section 1 to give a necessary and sufficient condition for the finiteness of the Galois cohomology of an arbitrary unipotent group over local function fields, see Theorem~\ref{theorem:main}. It says roughly that a unipotent group over a local function field  has a finite Galois cohomology set if and only if it has a decomposition series such that each factor is a unipotent group of Rosenlicht type. Section 3 deals with some calculations on unipotent groups of Rosenlicht type, and we also discuss about a question of Oesterl\'e.

We recall after Tits that a unipotent $k$-group $G$ is called $k$-{\it wound} if every $k$-homomorphism (or even, $k$-morphism as in \cite{KMT}) $\G_a\to G$ is constant. A polynomial $P:=P(x_1,\ldots,x_n)$ in $n$ variables $x_1,\ldots,x_n$ with coefficients in $k$ is said to be {\it  additive} if $P(x+y)=P(x)+P(y)$, for any two vectors $x\in k^n, y\in k^n$. If this is the case and if $k$ is infinite, $P$ is the so-called $p$-polynomial, i.e., $P(x_1,\ldots, x_n)=\sum_i\sum_j c_{ij}x_i^{p^j}$, a $k$-linear combination of $x_i^{p^j}$. If,  for each $i$, $m_i$ is the largest exponent $j$ for which $c_{ij}\neq 0$, then the {\it principal part } of $P$ is the polynomial $P_{\princ}:=c_{1m_1}x_1^{p^{m_1}}+\cdots+c_{nm_n}x_n^{p^{m_n}}$.

All algebraic groups considered in this paper are {\it linear} algebraic groups (in the sense of \cite{Bo}). In particular, they are smooth.

\section{Images of additive polynomials in local fields}
Let us first recall the notion of valuation independence (see \cite{DK} or \cite{Ku}). Let $(K,v)$  be a valued field, $L$ a subfield of $K$, and $(b_i)_{i\in I}$ a system of non-zero elements in $K$, with $I\not=\emptyset$. This system is called $L$-{\it valuation independent} if the following holds: for every choice of elements $a_i\in L$ such that $a_i\not=0$ for only finitely many $i\in I$, one has
$$v(\sum_{i\in I}a_ib_i)=\min_{i\in I}v(a_ib_i).$$
If $V$ is an $L$-subvector space of $K$, then this system is called a {\it valuation basis} of $V$ if it is a basis of $V$ and $L$-valuation independent.

\begin{lemma} 
\label{lemma:1}
Let $k$ be a field of characteristic $p>0$, $v$ a discrete valuation on $k$ whose value group is $\Z$. Let
 \[P(T)=\sum_{i=1}^rc_iT_i^{p^m}+\sum_{i=1}^r\sum_{j=1}^{m-1}c_{ij}T_i^{p^{m-j}}+ T_1\]
be a $p$-polynomial with coefficients in $k$, where $c_1,\ldots, c_r$ is a $k^{p^m}$-valuation basis of $k$. Take $a\in k$ and assume that the set $\{v(a-y)\mid y\in\im P\}$ admits a maximum. Then there exist two constants $\alpha\leq \beta$ depending only on $P$ (not depending on $a$) such that this maximum belongs to the segment $[\alpha,\beta]$ or $\infty$.   
\end{lemma}
\begin{proof}
Assume that $y_0\in k$ such that $v(a-y_0)$ is the maximum of $\{v(a-y)\,\mid\, y\in P(k\times \cdots\times k)\}$. After replacing $a$ by $a-y_0$ we can assume that $y_0=0$.
 
 Let $c_{10}:=c_1$ and $I$ the set of indices $j$, $0\leq j\leq m-1$, such that $c_{1j}\not=0$. We set 
\[\beta:=\max_{j\in I}\left\{\frac{-v(c_{1j})}{p^{m-j}-1}\right \}.\]

Any monomial  of $P(T_1,\ldots,T_r)-P_{\princ}(T_1,\ldots,T_r)$ is
of the form $\lambda T_j^{p^{m-s}}$, $\lambda\in k^*,1\leq j\leq r,s\geq 1$,
 and for such a monomial, we set
\[a(\lambda T_j^{p^{m-s}})=a_{\lambda,s,j}=\frac{v(\lambda)-v(c_j)}{p^{m}-p^{m-s}},\]
 and set 
\[\alpha :=\min_{\lambda,j,s}\{v(\lambda)+p^{m-s}a_{\lambda,s,j}\}.\]
 
It is trivial that $0\in I$, hence $\beta\geq \frac{-v(c_1)}{p^m-1}$. On the other hand, for the monomial $T_1$ in $P-P_{\princ}$, the constant
$a(T_1)=a_{1,m,1}=\frac{v(1)-v(c_1)}{p^m-1}=\frac{-v(c_1)}{p^m-1}$,
hence $\alpha\leq \frac{-v(c_1)}{p^m-1}\leq \beta$. 
 
We shall show that $\alpha\leq v(a)\leq \beta$ or $v(a)=\infty$.
 
a) Suppose that $\infty> v(a)>\beta$. Then for all $j\in I$, we have 
\[v(a)>\beta\geq \frac{-v(c_{1j})}{p^{m-j}-1},\]
and hence, 
\[v(c_{1j}a^{p^{m-j}})=v(c_{1j})+p^{m-j}v(a)>v(a).\]
Now we set
\[y:=\sum_{j\in I}c_{1j}a^{p^{m-j}}+a=P(a,0,\cdots,0)\in \im P.\]
Then we get
\[v(a-y)=v(\sum_{j\in I}c_{1j}a^{p^{m-j}})\geq \min_{j\in I} v(c_{1j}a^{p^{m-j}})>v(a),\]
a contradiction.

b) Now suppose that $v(a)<\alpha$. Since $c_1,\ldots,c_r$ is a $k^{p^m}$-valuation basis, we can write 
\[a=c_1a_1^{p^m}+\cdots+c_ra_r^{p^m},\]
 where $a_i\in k$ for all $i$ and 
\[v(a)=\min_i v(c_ia_i^{p^m}).\]

For any monomial $\lambda T_j^{p^{m-s}}$ of $P-P_{\princ}$, appearing in $P(T_1,..., T_r)$,
 $\lambda\in k^*,1\leq j\leq r,s\geq 1$, if  $v(a_j)<a_{\lambda,s,j}$
   then by the definition of $a_{\lambda,s,j}$, we have
\[v(\lambda a_j^{p^{m-s}})=v(\lambda)+p^{m-s}v(a_j)>v(c_j)+
p^{m}v(a_j)\geq v(a).\]

Also, if   $v(a_j)\geq a_{\lambda,s,j}$ then by the definition of $\alpha$, we have
\[v(\lambda a_j^{p^{m-s}})\geq v(\lambda)+p^{m-s}a_{\lambda,s,j}\geq \alpha > v(a) .\]
Thus, for all $j$, we always have
\[ v(\lambda a_j^{p^{m_j-s}}) > v(a).\]
Hence
\[ v\big(a-P(a_1,\ldots,a_r)\big)=v\big(P(a_1,...,a_r) -P_{\princ}(a_1,...,a_r)\big) > v(a),\]
a contradiction.
 \end{proof}
 On $\F_q((t))$, one has the natural valuation $v_t$ defined as: $v_t(\sum a_n t^n)=\inf\{n \mid a_n\not=0 \}$.
 
\begin{lemma}
\label{lemma:key}
 Let $k=\F_q((t))$ be a local field of characteristic $p>0$, with the natural valuation $v=v_t$. Let  $P(T_1,\ldots,T_r)$ be a separable $p$-polynomial with coefficients in $k$ and the principal part
$P_{\princ}(T_1,\ldots,T_r)=\sum_{i=1}^rc_iT_i^{p^m}.$
Assume that $c_1,\ldots,c_r$ is a $k^{p^m}$-basis of $k$ and $v(c_1),\ldots, v(c_r)$ are pairwise distinct modulo $p^m$. Then the quotient group $k/\im P$ is finite.    
 \end{lemma}
 
 \begin{proof}
 We write 
 \[P(T_1,\ldots,T_r)=\sum_{i=1}^rc_iT_i^{p^m}+(\cdots)+a_1T_1+\cdots+a_rT_r.\]
 Let $I:=\{i\mid a_i \not= 0 \}$. Then $I$ is a non-empty set since $P$ is separable. Since $v(c_1),\ldots, v(c_r)$ are pairwise distinct modulo $p^m$, there exists the unique index $i_0\in I$ such that 
\[ v(c_{i_0}a_{i_0}^{-p^m})=\max_{i\in I} v(c_i a_i^{-p^m}).\]
We may and shall assume that $i_0=1$. Let us change the variables $ X_1 :=a_1T_1+\cdots +a_rT_r$, and $X_i:=T_i$ for all $i>1$. Then the polynomial $P$  becomes
\[ Q(X_1,\ldots,X_r)=\sum_{i=1}^rb_iX_i^{p^m}+(\cdots)+X_1,\]
where $b_1= a_1^{-p^m}c_1$, $b_i =c_i-c_1 a_1^{-p^m}a_i^{p^m}$, for all $i>1$. Then $\im P=\im Q$. We also have
$v(b_1)=v(c_1)-p^mv(a_1)$, and $v(b_i)=v(c_i)$ for all $i\geq 2$. So $v(b_1),\ldots, v(b_r)$ are still pairwise distinct modulo $p^m$.  Then it is clear that $b_1,\ldots, b_r$ is a $k^{p^m}$-valuation basis of $k$.

 By \cite{DK}, we know that the images of additive polynomials in $(\F_q((t)),v_t)$ have the optimal approximation property, that is, for every $z\in k=\F_q((t))$ and every additive polynomial  $F(X_1,\ldots, X_n)$ with coefficients in $\F_q((t))$, the set  $\{v(z-y)\,\mid\, y\in \im F\}$ admits a maximum.

Let $\alpha,\beta$ be the constants depending only on $Q$ as in Lemma~\ref{lemma:1}. Take $a$ in $k$ and suppose that $a\not\in \im Q$. Then by Lemma~\ref{lemma:1}, there is $y_1\in \im Q$ such that $v(a-y_1)=m_1$, where $\alpha\leq m_1 \leq \beta$. There is some $j_1\in \F_q$ such that $v(a-y_1-j_1t^{m_1})=m^\prime_2,$  where $m^\prime_2>m_1$. Again by Lemma~\ref{lemma:1}  applying for $a-j_1t^{m_1}$, there is $y_2\in \im Q$ such that 
 \[ v(a-y_2-j_1t^{m_1})=m_2 \] 
is a maximum in $\{v(a-j_1 t^{m_1}-y) \mid y\in \im Q \}$ and  $m_1<m^\prime_2\leq m_2\leq \beta$, or $m_2=\infty$. Similarly, if $m_2\not=\infty$ then there are some $y_3\in \im Q$ and $j_2\in \F_q$ such that
 \[ v(a-y_3-j_1t^{m_1}-j_2t^{m_2})=m_3,\]
 where $m_3>m_2$ and $m_3\leq \beta$ or $m_3=\infty$. Denote by $[x]$ the largest integer not exceeding $x$. Since the segment $[\alpha,\beta]$ has at most $u:=[\beta-\alpha]+1$ natural numbers, there exist $y_{s+1}\in \im Q$ and $j_1,\ldots,j_s\in \F_q$, where $s\leq u$, and $m_1<\cdots<m_s$, with $m_i\in [\alpha,\beta]\cap\mathbb{N}$, such that
 \[v(a-y_{s+1}-j_1t^{m_1}-j_2t^{m_2}-\cdots-j_st^{m_s})=m_{s+1}\]
 is a maximum in $\{ v(a-j_1t^{m_1}-j_2t^{m_2}-\cdots-j_st^{m_s}-y)\mid y\in \im Q\}$ and $m_{s+1}>\beta.$
 Again by Lemma~\ref{lemma:1},  we have $m_{s+1}=\infty$ and  $a-y_{s+1}-j_1t^{m_1}-j_2t^{m_2}-\cdots-j_st^{m_s}\in \im Q$, hence
\[ a\in \sum_{j\in \F_q}\sum_{m\in [\alpha,\beta]\cap \mathbb{N}} jt^m+\im Q.\]
Therefore, $k/\im P=k/\im Q$ has at most $q([\beta-\alpha]+1)$ elements.
 \end{proof}
 
 \begin{lemma}[{\cite[Lemma 4]{DK}}] 
 \label{lemma:DK}
Let k be a local field of characteristic $p>0$, $P=f_1(T_1)+\cdots+f_r(T_r)$ an additive \textup{(}i.e. $p$-\textup{)} polynomial with  coefficients in $k$ in $r$ variables, the principal part of which vanishes nowhere over $k^r\setminus\{0\}$. Let $S=\im(P)=f_1(k)+\cdots+f_r(k)$. Then there are additive polynomials $g_1,\ldots,g_s\in k[X]$ in one variable $X$
such that
\begin{enumerate}
\item $S=g_1(k)+\cdots+g_s(k)$;
\item all polynomials $g_i$ have the same degree $d=p^\nu$, for some non-negative integer $\nu$;
\item the leading coefficients $b_1,\ldots, b_s$ of $g_1,\ldots, g_s$ are such that $v(b_1),\ldots, v(b_s)$ are distinct elements of $\{0,1,\ldots,d-1\}$.
\end{enumerate}
\end{lemma}
\begin{remark}
As in the proof of Lemma \ref{lemma:DK}, we may and we shall 
choose $d=\max\limits_i p^{m_i}$ 
and $s=\sum\limits_{i=1}^rd\cdot p^{-m_i},$ where $p^{m_i}=\deg f_i$.
\end{remark}

 \begin{lemma} 
\label{lemma:Rosenlicht type}
Let $k$ be a local function field of characteristic $p>0$. Let $P(T_1,\cdots, T_{r})$ be a separable $p$-polynomial in $r$ variables with coefficients in $k$ such that its principal part 
 \[ P_{\princ}=\sum_{i=1}^{r_1} c_i T_i^{p^M}+\sum_{i=r_1+1}^{r_1+r_2} c_i T_i^{p^{M-1}}+\cdots+\sum_{i=r_1+\cdots+r_{M-1}+1}^{r_1+\cdots+r_M} c_i T_i^{p} \]
 vanishes nowhere over $k^{r}\setminus \{0\}$. Then we  have the following
 \[ r_1+ pr_2+\cdots+p^{M-1}r_M\leq p^M, \]
 and the equality holds if and only if the quotient group $k/\im P$ is finite.
 \end{lemma}
 
\begin{proof}
We write $$P=f_1(T_1)+\cdots+f_{r}(T_r),$$ 
where each $f_i$ is a $p$-polynomial in one variable $T_i$ with coefficients in $k$ and of degree $p^{m_i}$. Let 
\[ S=\im (P)= f_1(k)+\cdots+f_r(k).\]
Choose $g_1,\ldots,g_s$ with leading coefficients $b_1,\ldots,b_s$, with $d:= \max\limits_i p^{m_i}$, $s:=\sum\limits_{i=1}^rd\cdot p^{-m_i}$  as in Lemma~\ref{lemma:DK}. Then 
\[ d=p^M \text{ and } s=r_1+ pr_2+\cdots+p^{M-1}r_M.\]
Let 
\[ Q(T_1,\ldots,T_s)=g_1(T_1)+\cdots+g_s(T_s).\]
Since $v(b_1),\ldots, v(b_s)$ are distinct elements of $\{0,1,\ldots,d-1\}$, the principal part of $Q$ vanishes nowhere over $k^s\setminus\{0\}$ and $b_1,\ldots,b_s$ is a $k^{p^M}$-linearly independent subset of $k$. Hence
\[ s=r_1+ pr_2+\cdots+p^{M-1}r_M \leq p^M.\]

Assume that $s<p^M$. Then by \cite[Proposition 4.5]{TT}, $k/\im P=k/\im Q$ is infinite. Now assume that $s=p^M$. Then by Lemma \ref{lemma:key}, $k/\im P$ is finite. Therefore $k/\im P$ is finite if and only if $s=p^M$.
\end{proof}

\begin{remark}
With notations as in Lemma~\ref{lemma:Rosenlicht type}, one has $r_1+\cdots+r_M=r$.
\end{remark}

Lemma~\ref{lemma:Rosenlicht type} motivates the following definition. 
 \begin{definition}
\label{def:Rosenlicht type}
 Let $k$ be a non-perfect field of characteristic $p>0$, $P(T_1,\ldots,T_{d+1})$ a separable $p$-polynomial with coefficients in $k$. The polynomial $P$ is called of {\it Rosenlicht type} if its principal part (after reindexing the variables)
 \[ P_{\princ}=\sum_{i=1}^{r_1} c_i T_i^{p^M}+\sum_{i=r_1+1}^{r_1+r_2} c_i T_i^{p^{M-1}}+\cdots+\sum_{i=r_1+\cdots+r_{M-1}+1}^{r_1+\cdots+r_M} c_i T_i^{p} \]
 vanishes nowhere on $k^{d+1}\setminus \{0\}$ and 
  \begin{equation*}
 \begin{cases} 
r_1+r_2+\cdots +r_M=d+1,\\
r_1+ pr_2+\cdots+p^{M-1}r_M= p^M,\\
r_1\geq 1, r_i\geq 0, i>1.
\end{cases}
\end{equation*}
A unipotent $k$-group $G$ is called of {\it Rosenlicht type} if $G$ is $k$-isomorphic to a $k$-subgroup of $\G_a^{d+1}$, where $d=\dim G$, which is defined as the kernel of a separable $p$-polynomial $P(T_1,\ldots,T_{d+1})\in k[T_1,\ldots,T_{d+1}]$ of Rosenlicht type.  
\end{definition}

\begin{corollary} 
\label{cor:Rosenlicht type}
Let $k$ be a local field of characteristic $p>0$, $G$ be a commutative unipotent $k$-wound $k$-group, which is killed by $p$. Then $\H^1(k,G)$ is finite if and only if $G$ is of Rosenlicht type. 
\end{corollary}

\begin{proof} Let $\dim G=d$. By Tits theory on unipotent groups, $G$ is $k$-isomorphic to a subgroup of $\G_a^{d+1}$, which is the zero set  of  a separable $p$-polynomial $P$ in  ${d+1}$ variables with coefficients in $k$, the principal part of which vanishes nowhere over $k^{d+1}\setminus \{0\}$ (see \cite[III, 3.3.6]{Ti}, \cite[V, 6.3, Proposition]{Oe} or \cite[Proposition B.1.13]{CGP}). From the exact sequence of $k$-groups 
 \[ 0\to G\to \G_a^{d+1} \stackrel{P}{\to} \G_a\to 0,\]
one has $\H^1(k,G)=k/\im P$. By the very definition of group of Rosenlicht type and Lemma~\ref{lemma:Rosenlicht type}, $\H^1(k,G)$ is finite if and only if $G$ is of Rosenlicht type. 
\end{proof}
 Over local field of characteristic $p>0$, the property of a wound unipotent group being Rosenlicht type   does not depend on the choice of its defining $p$-polynomials. More precisely, one has

\begin{corollary}
\label{cor:independent}
Let $k$ be a local field of characteristic $p>0$ and $G$ a $k$-unipotent group of Rosenlicht type of dimension $d$. Assume that $G$ is $k$-isomorphic to a $k$-subgroup of $\G_a^{d+1}$, which is  defined as the kernel of a separable $p$-polynomial $P(T_1,\ldots,T_{d+1})\in k[T_1,\ldots,T_{d+1}]$ whose principal part vanishes nowhere over $k^{d+1}\setminus\{0\}$. Then the polynomial $P$ is of Rosenlicht type.  
\end{corollary}
\begin{proof}
This follows immediately from Lemma~\ref{lemma:Rosenlicht type} and the proof of Corollary~\ref{cor:Rosenlicht type}.
\end{proof}

 \section{Main theorem}
 We first recall some basic results of the theory of unipotent groups over an arbitrary field (see \cite{Ti}, \cite[Chapter V]{Oe} and \cite[Appendix B]{CGP}). A unipotent algebraic group $G$ over a field $k$ of characteristic $p>0$ is called $k$-{\it split} if it admits a composition series by $k$-subgroups with successive quotients $k$-isomorphic to $\G_a$. We say that $G$ is $k$-{\it wound} if any $k$-morphism of affine groups $\G_a\to G$ is constant. For any unipotent group $G$ defined over $k$, there is a maximal $k$-split $k$-subgroup $G_s$, which enjoys the following properties: it is normal in $G$, the quotient $G/G_s$ is $k$-wound and the formation of $G_s$ commutes with separable (not necessarily algebraic) extensions, see \cite[Chapter V, 7]{Oe} and \cite[Appendix B, B.3]{CGP}. The group $G_s$ is called the $k$-{\it split part} of $G$.
 
Let $G$ be a unipotent group over $k$. Then there exists a maximal central connected $k$-subgroup of $G$ which is annihilated by $p$. This group is called {\it cckp-kernel} of $G$ and denoted by $cckp(G)$ or $\kappa(G)$. Here $\dim (\kappa(G))>0$ if $G$ is not finite. The formation of $\kappa(G)$ commutes with any separable extension on $k$, see \cite[Appendix B, B.3]{CGP}. 

The following statements are equivalent:
\begin{enumerate}
\item[(i)] $G$ is wound over $k$,
\item[(ii)] $\kappa(G)$ is wound over $k$.
\end{enumerate}

If these conditions are satisfied then $G/\kappa(G)$ is also wound over $k$ (see \cite[Chapter V, 3.2]{Oe}, \cite[Appendix B, B.3]{CGP}).

We set
\[
\begin{aligned}
\kappa^1(G)&:=\kappa(G), \kappa_1(G):=G/\kappa^1(G)\\
\kappa^2(G)&:=\kappa(\kappa_1(G)), \kappa_2(G):=\kappa_1(G)/\kappa^2(G)\\
\ldots \\
\kappa^{n+1}(G)&:=\kappa(\kappa_n(G)), \kappa_{n+1}(G):=\kappa_n(G)/\kappa^{n+1}(G)\\
\end{aligned}
\] 
We call $\kappa^{1}(G), \kappa^{2}(G), \ldots, \kappa^{n}(G),\ldots $ the {\it cckp-kernel series} of $G$. Then there exists $n$ such that $\kappa^{n}(G)=0$, and the least such number will be called the {\it cckp-kernel length} of $G$ and denoted by $lcckp(G)$. The formation of $\kappa^i(G)$ commutes with any separable extension on $k$.

We now recall the following result of Oesterl\'e (see \cite[Chapter IV, 2.2]{Oe}).
\begin{lemma} 
\label{lemma:surjective}
Let $G$ be a linear algebraic group defined over a field $k$ and $U$ a normal unipotent algebraic subgroups of $G$ defined over $k$. Then the canonical map 
\[ \H^1(k,G)\to \H^1(k,G/U) \]
is surjective.
\end{lemma}

\begin{proposition}
\label{prop:finite}
 Let $G$  be a smooth connected unipotent group which is defined and wound over a local function field $k$ of characteristics $p>0$. Let $\kappa^1(G), \kappa^2(G), \ldots, \kappa^{n}(G)=0$, $n=lcckp(G)$, be its cckp-kernel series. Then $\H^1(k,G)$ is finite if and only if $\H^1(k, \kappa^i(G))$ is finite for all $i$.
 \end{proposition} 
 \begin{proof}
Suppose that $\H^1(k,G)$ is finite. From the exact sequence
 \[1\to \kappa^1(G)\to G \to \kappa_1(G)\to 1,\]
 we derive the  following exact sequence
 \[\kappa_1(G)(k)\stackrel{\delta}{\to} \H^1(k,\kappa^1(G))\to \H^1(k,G) \to \H^1(k,\kappa_1(G)).\]
We endow $(\kappa_1(G))(k)$ and $\H^1(k,\kappa^1(G))$ with the topology induced from the natural topology (the valuation topology) on $k$; then the map $\delta$ is continuous. 
Since $\kappa_1(G)$ is $k$-wound, $(\kappa_1(G))(k)$ is compact by \cite[Chapter VI, Section 1]{Oe}. On the other hand, since $\kappa^1(G)$ is commutative, $k$-wound and killed by $p$, $\kappa^1(G)$ is $k$-isomorphic to a $k$-subgroup of $\G_a^{d+1}$ which is given as the kernel of a separable $p$-polynomial $F$ in $d+1$ variables, where $d=\dim(\kappa^1(G)) $ and $F$ is considered as a homomorphism $F:\G_a^{d+1}\to \G_a$. One thus finds that  $\H^1(k,\kappa^1(G))\simeq k/F(k^{d+1})$. 
One checks that the subgroup $F(k^{d+1})\subset k$ is open, since $F$ is a separable morphism and we may use the implicit function theorem (see [Se2]) in this case. Hence the topology on $\H^1(k,\kappa^1(G))$ is discrete. Since the map $\delta$ is continuous, its image $\im (\delta)$ is compact in the discrete topological group $\H^1(k,\kappa^1(G))$. Therefore $\im (\delta)$ is finite and by twisting argument (see \cite[Chapter I, Section 5.4, Corollary 3]{Se1}), and the finiteness assumption of $\H^1(k,G)$, we get that $\H^1(k,\kappa^1(G))$ is finite.  
We also know that the natural map $\H^1(k,G) \to \H^1(k,\kappa_1(G))$ is surjective  by Lemma~\ref{lemma:surjective}, so $\H^1(k,\kappa_1(G))$ is also finite.

Similarly, by replacing $G$ by $\kappa_1(G)$ then we can show that $\H^1(k,\kappa^2(G))$ is finite, since $\kappa^2(G)=\kappa(\kappa_1(G))$ by definition. Inductively, we can prove that $\H^1(k,\kappa^i(G))$ is finite for all $i$.

Conversely, assume that $\H^1(k,\kappa^i(G))$ is finite for all $i$. Since $0=\kappa^n(G)=\kappa(\kappa_{n-1}G)$ and $\kappa_{n-1}G$ is connected, $\kappa_{n-1}G$ is trivial. Hence from the exact sequence
\[ 1\to\kappa^{n-1}G\to \kappa_{n-2}G\to\kappa_{n-1}G\to 1,\]
and the assumption that $\H^1(k,\kappa^{n-1}(G))$ is finite, we deduce that $\H^1(k,\kappa_{n-2}(G))$ is finite. Inductively, we can prove that $\H^1(k,\kappa_{n-2}(G))$,$\ldots$, $\H^1(k,\kappa_1(G))$ are finite. Then $\H^1(k,G)$ is also finite.  
 \end{proof}
 
We now have the following main result of this paper. 

\begin{theorem} 
\label{theorem:main}
Let $k$ be a local field of characteristic $p>0$, $G$ a smooth unipotent group defined over $k$. Let $G_s$ be the $k$-split part of $G$. Then $\H^1(k,G)$ is finite if and only if $G$ is connected and $\kappa^i(G/G_s)$ is of Rosenlicht type for all $i$.
\end{theorem}

\begin{proof}
Assume that $\H^1(k,G)$ is finite. Let $G^0$ be the connected component of $G$, then $G/G^0$ is of dimension 0. Since the natural map $\H^1(k,G)\to \H^1(k,G/G^0)$ is surjective, $\H^1(k,G/G^0)$ is finite. Then by \cite[Proposition 4.7]{TT}, $G/G^0$ is trivial and $G$ is connected. Also, since the natural map $\H^1(k,G)\to \H^1(k,G/G_s)$ is surjective,  $\H^1(k,G/G_s)$ is finite. Hence by Proposition~\ref{prop:finite}, $\H^1(k,\kappa^i(G/G_s))$ is finite for all $i$ since $G/G_s$ is connected and wound over $k$. By Corollary~\ref{cor:Rosenlicht type}, $\kappa^i(G/G_s)$ are of Rosenlicht type since such groups are commutative, $k$-wound and killed by $p$.

Conversely, assume that $G$ is connected and for all $i$, $\kappa^i(G/G_s)$ is of Rosenlicht type. Then by Corollary~\ref{cor:Rosenlicht type}, $\H^1(k,\kappa^i(G/G_s))$ is finite for all $i$. By Proposition~\ref{prop:finite}, $\H^1(k,G/G_s)$ is finite. Since $G_s$ is $k$-split, the natural map $\H^1(k,G)\to \H^1(k,G/G_s)$ is bijective by \cite[Lemma 7.3]{GiMB}. Therefore $\H^1(k,G)$ is finite. 
\end{proof}

\begin{remarks} (1) The unipotent group $G=\{(x,y)\in \G_a\times \G_a \mid y^p-tx^p-x=0\}$ as in Oesterl\'e's example in the Introduction is connected $k$-wound, commutative, killed by $p$ and  of dimension 1. So $\kappa^1(G)=G$  and $\kappa^2(G)=0$.  Using a simple calculation, one can show that the dimension of a unipotent group of Rosenlicht type is always divisible by $p-1$ (see Corollary~\ref{cor:divisible} below). Therefore, $G=\kappa^1(G)$ is not of Rosenlicht type since $\dim G=1<p-1$ (note that we assume that $p>2$ in Oesterl\'e's example). So $\H^1(k,G)$ is infinite in the light of Theorem~\ref{theorem:main}.

(2) If   one allows $p=2$ in Oesterl\'e example in the Introduction, namely, if we consider the case $k=\F_q((t))$ of characteristic $2$ and $G=\{(x,y)\in \G_a\times \G_a \mid y^2-tx^2-x=0\}$, then $\H^1(k,G)$ is finite. Because in this case, $G=\kappa^1(G)$ is of Rosenlicht type, $\kappa^2(G)=0$ and $G$ satisfies the conditions in Theorem~\ref{theorem:main}. One can even compute that the cardinality of $\H^1(k,G)$ is 2 (see \cite[Proposition 4.2(a)]{TT}).

Let $k=\F_q((t))$ be a local field of characteristic 2. To get an example of $G$  over $k$, which has infinite Galois cohomology, one may consider an example of Serre (see \cite[Chapter III, Section 2, Exercise 3, page 130]{Se1}). Namely, take $G=\{(x,y)\in \G_a\times \G_a \mid y^2+y+tx^4=0\}$, then since the polynomial $y^2+y+tx^4$ is not of Rosenlicht type, $\kappa^1(G)=G$ is not of Rosenlicht type by Corollary~\ref{cor:independent}.  Therefore $\H^1(k,G)$ is infinite.

(3) After writing an early version of this paper, we learned that in \cite{CGP}, for a (smooth) \emph{wound} unipotent group $G$ over a field $k$, the authors also define the ascending chain of (smooth) connected normal $k$-subgroups of $G$ as follows: $G_0=0$ and $G_{i+1}/G_{i}$ is the cckp-kernel of the $k$-wound group $G/G_i$ for all $i\geq 0$. These subgroups are stable under $k$-group
automorphisms of $G$, their formation commutes with any separable extension of $k$, and $G_i = G$ for
sufficiently large $i$. See \cite[Corollary B.3.3]{CGP}.

The relation between this ascending chain of subgroups and our cckp-series for a wound unipotent $k$-group $G$  is that
$G_1/G_0=\kappa^1(G)$ and $G_i/G_{i-1}=\kappa^i(G)$ for all $i\geq 1$. And the main theorem (Theorem~\ref{theorem:main}) can be restated as follows.

\begin{thma}
Let $k$ be a local field of characteristic $p>0$, $G$ a smooth unipotent group defined over $k$. Let $G_s$ be the $k$-split part of $G$ and define $U=G/G_s$. Define the ascending chain of (smooth) connected normal $k$-subgroups of $U$ as follows: $U_0=0$ and $U_{i+1}/U_{i}$ is the cckp-kernel of the $k$-wound group $U/U_i$ for all $i\geq 0$. Then $\H^1(k,G)$ is finite if and only if $G$ is connected and $U_{i+1}/U_i$ is of Rosenlicht type for all $i\geq 0$.
\end{thma}
\end{remarks}

\section{Unipotent groups of Rosenlicht type and a question of Oesterl\'e}
In this section we shall make some calculations on unipotent groups of Rosenlicht type of small dimension and give some discussions about a question of Oesterl\'e, see Question~\ref{question:Oe} below.
\subsection{Some calculations}
First we have the following result concerning the dimension of unipotent groups of Rosenlicht type.
\begin{corollary}
\label{cor:divisible}
Let $G$ be a unipotent algebraic group over a non-perfect field $k$ of characteristic $p>0$. If $G$ is of Rosenlicht type then $\dim(G)$ is divisible by $p-1$.
\end{corollary}
\begin{proof}Let $\dim G=d$ and $G$ is defined by a separable $p$-polynomial $P$ as in Definition 5. Then for some integers  $M>0$, $r_1>0$, $r_2,\ldots,r_M\geq 0$, we have 
\begin{gather} 
r_1+r_2+\cdots +r_M=d+1\tag{3.1} \\
r_1+ pr_2+\cdots+p^{M-1}r_M= p^M \tag{3.2}
\end{gather}
By subtracting (3.1) from (3.2), we get $(p-1)r_2+\cdots+(p^{M-1}-1)r_M=(p^M-1)-d$. This yields that $p-1$ divides $d$ since $p^i-1$ is divisible by $p-1$ for all $i$.
\end{proof}

For $d=k(p-1)$, we want to solve the following equations with integer variables $M,r_1\geq 1$, $r_2,\ldots,r_M\geq 0$: 
\begin{gather} 
r_1+r_2+\cdots +r_M=k(p-1)+1 \tag{3.3}\\
r_1+ pr_2+\cdots+p^{M-1}r_M= p^M \tag{3.4}
\end{gather}
Equation (3.4) yields that $p$ divides $r_1$, and then $r_1=\ell_1p$, for some integer $\ell_1\geq 1$. By substituting $r_1=\ell_1p$ in (3.4), we have $\ell_1+r_2+\cdots+p^{M-2}r_{M}=p^{M-1}$. Then $\ell_1+r_2=l_2p$, for some integer $\ell_2\geq 1$. Similarly, there are natural numbers $\ell_3,\ldots,\ell_M\geq 1$ such that
\begin{equation}
r_1=\ell_1p, \ell_1+r_2=\ell_2p, \ell_2+r_3=\ell_3p,\ldots, \ell_{M-1}+r_M=\ell_Mp.\tag{3.5}
\end{equation}
By substituting (3.5) in (3.3)-(3.4), we get     
\begin{equation*}
 \begin{cases} 
\ell_1+\ell_2+\cdots +\ell_M=k\\
\ell_M=1
\end{cases}
\end{equation*}
For example: 

If $k=1$ then $M=1$, $\ell_1=1$. 

If $k=2$ then $M=2$, $\ell_1=\ell_2=1$. 

If $k=3$ then $M=2$, $\ell_1=2, \ell_2=1$ or $M=3$, $\ell_1=\ell_2=\ell_3=1$.

From the above calculations, we have the following proposition. 
\begin{proposition}
\label{prop:calculation}
 Let $k$ be a non-perfect field of characteristic $p>0$. 
 \begin{enumerate}
 \item[(a)] Every unipotent $k$-groups of Rosenlicht type  of dimension $p-1$ \textup{(}resp. $2(p-1)$\textup{)} is $k$-isomorphic to a $k$-subgroup of $\G_a^p$ \textup{(}resp. $\G_a^{2p-1}$\textup{)} defined as the kernel of a separable $p$-polynomial $P(T_1,\ldots,T_p)$ \textup{(}resp. $P(T_1,\ldots,T_{2p-1})$\textup{)} with the principal part of the form
\[ P_{\princ}=c_1T_1^p+\cdots+c_pT_p^p, \]
\[ (\text{resp. } P_{\princ}=c_1T_1^{p^2}+\cdots+c_pT_p^{p^2}+c_{p+1}T_{p+1}^p+\cdots+c_{2p-1}T_{2p-1}^p,)\]
which  vanishes nowhere over $k^p \setminus\{0\}$ \textup{(}resp. $k^{2p-1} \setminus\{0\}$\textup{)}. 
\item[(b)] Every unipotent $k$-groups of Rosenlicht type  of dimension $3(p-1)$ is $k$-isomorphic to a $k$-subgroup of $\G_a^{3p-2}$ defined as the kernel of a separable $p$-polynomial $P(T_1,\ldots,T_{3p-2})$ with the principal part of the form
\[ P_{\princ}=c_1T_1^{p^3}+\cdots+c_{2p}T_{2p}^{p^3}+c_{2p+1}T_{2p+1}^p+\cdots+c_{3p-2}T_{3p-2}^p,\]
or of the form
\[ P_{\princ}=c_1T_1^{p^3}+\cdots+c_pT_p^{p^3}+c_{p+1}T_{p+1}^{p^2}+\cdots+c_{2p-1}T_{2p-1}^{p^2}+c_{2p}T_{2p}^p+\cdots+c_{3p-2}T_{3p-2}^p,\]
which  vanishes nowhere over $k^{3p-2} \setminus\{0\}$. 
\end{enumerate}
\end{proposition}

\subsection{Oesterl\'e's construction} 
\label{subsec:Oe's construction}
We now recall Oesterl\'e's construction associating a torus defined over a non-perfect field of positive characteristic with a smooth unipotent group defined  and wound over that field (see \cite[Chapter VI, 5]{Oe}). Let $k$ be a field of characteristic $p>0$. Let $T$ be a $k$-torus, $k^\prime$ a finite purely inseparable extension of $k$ of degree $p^n$. Denote by $G=\prod_{k^\prime/k}(T\times_k k^\prime)$, the Weil  restriction from $k^\prime$ to $k$ of $T$ where $T$ is considered as an algebraic group over $k^\prime$. Then  $G$ is connected and commutative and $T$ is a maximal torus of $G$. Denote by $U(T,k,k^\prime)$ (or simply by $U$)  the quotient group $G/T$, then $U$ is a $k$-wound unipotent group. We show that if $k$ is a local or global function field then the groups in the cckp-series of $U$, $\kappa^i(U)$ are of Rosenlicht type. 

\begin{proposition}
\label{prop:Oe's construction}
  Let $k$ be a local (or global) field of characteristic $p>0$, $k^\prime$ a finite purely inseparable extension of $k$, $T$ a $k$-torus. Let $U$ be the unipotent group associated with $T$ as above. Then the groups $\kappa^i(U)$ are of Rosenlicht type.
  
  In particular, $U$ has a subgroup of Rosenlicht type.
 \end{proposition}
 
 \begin{proof} First we prove the proposition for the case $k$ is a local field. 
 Let $[k^\prime:k]=p^n$. Since $U$ is connected and wound over $k$, $\kappa^i(U)$ are all of Rosenlicht type if and only if $\H^1(k,U)$ is finite by Proposition~\ref{prop:finite}, which is in turn equivalent to the fact that both groups $\H^1(k,T)/p^n$ and $_{p^n}\H^2(k,T)$ are finite (see \cite[Proposition 5.1]{TT}), where for an abelian group $A$ and a natural number $n$, $A/n$ (resp. $_nA$) is the cokernel (resp. kernel) of the natural endomorphism $A\to A, x\mapsto nx$. 
 
Let $X(T)=\Hom(T,\G_m)$ be the character group of $T$. Let $\H^0(k, X(T))^{\wedge}$ be the completion of the abelian group $\H^0(k,X(T))$ for the topology given by subgroups of finite index. Then by the Tate-Nakayama duality (see \cite[Chapter I, Corollary 2.4]{Mi}), there is a duality between the compact group ${\H^0(k, X(T))}^{\wedge}$ and the discrete group $\H^2(k,T)$. Furthermore, the group $\H^1(k,T)$ is finite. Since $X(T)$ is a free abelian group of finite rank, we deduce that the group $_{p^n}\H^2(k,T)$ is finite. Hence $\H^1(k,U)$ is finite and  $\kappa^i(U)$ are all of Rosenlicht type.

Now we prove the proposition for the case $k$ is a global field. Let  $v$ be a place of  the global field $k$ and let $k_v$ be the completion of $k$ at $v$.  Then $k_v$ is a separable extension of $k$, see the proof of \cite[Chapter VI, 2.1, Proposition]{Oe}.
By the local case above, the groups in the cckp-kernel series of $U \times_k k_v$ (the base change of $U$ from $k$ to $k_v$) are of Rosenlicht type. Hence the groups $\kappa^i(U)$ are also of Rosenlicht type since the property of being of Rosenlicht type is unchanged under separable extentions and the formation of $\kappa^i(U)$ commutes with any separable extension on $k$.
\end{proof}

\begin{corollary}
Notations being as in Proposition~\ref{prop:Oe's construction}, if $\dim T=1$, $[k^\prime:k]=p$, then $U$ is $k$-isomorphic to a subgroups of $\G_a^p$ defined as the kernel of a $p$-polynomial of the form
$c_1T_1^p+\cdots+c_pT_p^p+T_p$, where $c_1,\ldots,c_p$ is a $k^p$-basis of $k$.
\end{corollary}
\begin{proof}
One has $\dim U=p-1$. By Proposition~\ref{prop:Oe's construction} and Corollary~\ref{cor:divisible}, $U$ is itself of Rosenlicht type. By Proposition~\ref{prop:calculation}, $U$ is $k$-isomorphic to a subgroups of $\G_a^p$ defined as the kernel of a $p$-polynomial of the form
$c_1T_1^p+\cdots+c_pT_p^p+aT_p$, where $c_1,\ldots,c_p$ is a $k^p$-basis of $k$, and by changing variables we can take $a=1$. 
\end{proof}

\begin{remark}
Let $k$ be a field of characteristic $p>0$ such that $[k^{1/p}:k]=p$, and $t$ an element in  $k-k^p$. Then, with $T=\G_m$ and $k^\prime=k^{1/p}$, Oesterl\'e shows explicitly that $U=U(T,k,k^\prime)$ is $k$-isomorphic to a $k$-subgroup of $\G_a^p$ defined by the equation $x_0^p+tx_1^p+\cdots+t^{p-1}x^{p}_{p-1}=x_{p-1}$. 
\end{remark}

We would like to ask whether all unipotent groups of Rosenlicht type over a local or global field are obtained by using Oesterl\'e's construction (with a suitable torus and a suitable finite purely separable extension $k^\prime/k$) as in Proposition~\ref{prop:Oe's construction}. Namely, we have
\begin{question} 
\label{question:4}
Is any unipotent group of Rosenlicht type over a local (or global) field arisen as a group in the $cckp$-series of a unipotent group constructed by Oesterl\'e as in Proposition~\ref{prop:Oe's construction} (associated with the  purely inseparable Weil restriction of a torus)?
\end{question}

 \subsection{Oesterl\'e's question}
If  $k$ is a global function field and $\dim T>0$ then the group $U$ constructed as in Subsection~\ref{subsec:Oe's construction} is $k$-unirational, in particular, the group of $k$-rational points $U(k)$ is infinite (see \cite[Chapter VI, 5.1, Lemma]{Oe}). Oesterl\'e even raised  the following question (see \cite[page 80]{Oe}):

\begin{question}[Oesterl\'e]
\label{question:Oe}
If  a wound unipotent group $G$ has an infinite number of rational points over a global function field $K$,
\begin{enumerate}
 \item Does $G$ have a subgroup defined over $K$ of dimension $\geq 1$ such that its underlying variety is $K$-unirational? 
\item Better, does $G$ have a subgroup $U$ of the type as  above (associated with the purely inseparable Weil restriction of a torus)?
\end{enumerate}
\end{question}

In seeking an answer for  Question~\ref{question:Oe} part (2) and by looking at Proposition~\ref{prop:Oe's construction}, the following question is arisen quite naturally. 
\begin{question}
\label{question:2}
 Let $K$ be a global function field, $G$ a $K$-wound unipotent $K$-group. Assume that $G(K)$ is infinite. Is it true that $G$ has a subgroup $H$ of Rosenlicht type?    
\end{question}
An affirmative answer for the part (2) of Question~\ref{question:Oe} could give rise to an affirmative answer for Question~\ref{question:2} by Proposition~\ref{prop:Oe's construction}. (And of course, if one could find a counter-example for Question~\ref{question:2}, one then also finds a counter-example for Question~\ref{question:Oe}, part (2).)

We conclude the paper by raising one more question concerning unipotent groups of Rosenlicht type, which is hopefully related to Oesterl\'e's question.

\begin{question} 
\label{question:3}
Let $G$ be a unipotent group of Rosenlicht type over a global function field $K$ of positive characteristic. Is the group of $K$-rational points $G(K)$ infinite? Is the underlying variety of $G$ is $K$-unirational?
\end{question}

\begin{remarks}
(1 ) If the answer for Questions~\ref{question:2} and \ref{question:3} are both YES, then the answer for Question~\ref{question:Oe}, part (1) is also YES. We also note that \cite[Chapter VI, 3.4, Proposition]{Oe} shows that Questions~\ref{question:2} and \ref{question:3} both have positive answers when the characteristic of $K$ is 2 and $\dim G=1$.

(2) If the answer for Questions~\ref{question:4} and \ref{question:2} are both YES, then we can give a partial  answer for Question~\ref{question:Oe}, part (2). Namely, notations being as in these questions, then we can show that $G$ has a non-trivial subgroup of the type  $\kappa^i(U)$ for some $U$ of the type as in Proposition~\ref{prop:Oe's construction}.

(3) Questions~\ref{question:2} and \ref{question:3} deal with unipotent groups of Rosenlicht type, which are defined by "concrete" equations. So we hope that it is easier to find answers for these questions than to find an answer for Oesterl\'e's question (Question~\ref{question:Oe})
\end{remarks}

{\bf Acknowledgments:} We are pleased to thank Nguyen Quoc Thang for his careful guidance and encouragement. This paper contains some results presented in a talk given by the author in the seminar Vari\'et\'es Rationnelles at the Ecole Normale Sup\'erieure. We thank Phillipe Gille for remarks and discussions related to the results presented here. We would like to thank H\'el\`ene Esnault and Eckart Viehweg for their support  and stimulus. We would like to thank the referee for his/her very useful comments and corrections, which helps us to improve the paper.

\end{document}